\documentclass[12pt]{amsart}
\usepackage{amssymb,amsfonts,latexsym,amsmath,amsthm,graphicx}
\usepackage{parskip}
\usepackage{mathrsfs}
\usepackage{mathtools}

\mathtoolsset{showonlyrefs,showmanualtags}

\usepackage{esint}

\usepackage{hyperref}

\usepackage[margin=1.16in]{geometry}

\numberwithin{equation}{section}
\newtheorem{thm}{Theorem}[section]

\newtheorem{cor}[thm]{Corollary}
\newtheorem{lemma}[thm]{Lemma}

\newtheorem{q}[thm]{Question}
\theoremstyle{remark} 
\newtheorem{remark}[]{Remark}

\newcommand{\be}{\begin{equation}}
\newcommand{\ee}{\end{equation}}

\newcommand{\EE}{\mathbb{E}}
\newcommand{\PP}{\mathbb{P}}
\newcommand{\RR}{\mathbb{R}}

\DeclareMathOperator{\Real}{Re}
\DeclareMathOperator{\sech}{sech}
\DeclareMathOperator{\csch}{csch}
\DeclareMathOperator{\Corr}{Corr}

\newcommand{\ZNomega}{Z_{N,\omega}}

\usepackage[colorinlistoftodos]{todonotes}

\newcommand{\e}{\varepsilon}

\begin{document}

\title[Random Melnikov functions]{Zeros of Melnikov functions for periodically perturbed Hamiltonian systems with random coefficients}
\date{}

\author[A. Ittayem]{Ali Ittayem}
\address{Department of Mathematical Sciences, Florida Atlantic University,
Boca Raton, FL 33431, USA}
\email{aittayem2015@fau.edu}

\author[E. Lundberg]{Erik Lundberg}
\address{Department of Mathematical Sciences, Florida Atlantic University,
Boca Raton, FL 33431, USA}
\email{elundber@fau.edu}

\begin{abstract}
Introducing a probabilistic perspective on a classical theme within the study of chaotic dynamics, we investigate the zeros of the Melnikov distance function associated to a periodically perturbed Hamiltonian system when the perturbation has random coefficients.  We provide asymptotic estimates (as the period of the perturbation becomes large or small) for the average number of zeros and for the probability of no zeros.
\end{abstract}

\maketitle

\section{Introduction}

In this paper, we consider the nonlinear differential equation with a periodic external force $f(t)$
\begin{equation}\label{eq:DuffingSingle}
\ddot{y} +  \e \delta\dot{y} -  y + y^3 
=
\e f(t).
\end{equation}
This equation is a classical example of a nonlinear oscillator exhibiting chaotic dynamics \cite{GuckHolmes}.  A central tool in this analysis is the associated \emph{Melnikov function} (introduced in \cite{Melnikov} and used to study complicated dynamics in \cite{Arnold}, \cite{HolmesStrange}) that measures the first-order (in $\e$) separation between the stable and unstable manifolds of a hyperbolic fixed point of the associated Poincar\'e map.  
Simple zeros of the Melnikov function detect, for sufficiently small $\e>0$, transverse intersections of the stable and unstable manifolds near the unperturbed homoclinic orbit, which in turn lead to homoclinic chaos by the Smale--Birkhoff theorem.

Our goal is to analyze the typical behavior of the Melnikov function when the forcing function $f$ is an arbitrary trigonometric polynomial of degree $N$,
\be
f(t) = \sum_{k=1}^N a_k \cos(\omega k t)+ b_k \sin(\omega k t).
\ee
Previous studies of the Melnikov function have primarily focused on low-dimensional parameter spaces of differential equations, with special choices for the periodic forcing, such as a single sinusoid \cite{GuckHolmes}, a sum of two sinusoids \cite{Yang}, \cite{Wang}, \cite{Gear}, elliptic functions \cite{Chacon}, \cite{Li2024}, and certain piecewise linear waveforms such as triangular waves \cite{PiecewiseLinear}.

In order to investigate high-dimensional parameter space, we introduce a probabilistic perspective\footnote{The dynamical system remains deterministic with this approach.  It is merely the analysis of parameter space that becomes probabilistic, i.e., we are investigating the typical outcome of a randomly sampled deterministic system.}, equipping parameter space with a probability measure.
More precisely, we consider the following question while sampling the parameters as i.i.d. (independent and identically distributed) mean-zero Gaussians.

\begin{remark}
This approach complements qualitative methods (e.g., transversality and stability arguments) that are commonly used to establish genericity and robustness in dynamics.  Establishing that a
phenomenon occurs on a nonempty open set of parameters shows that it is robust and has positive probability (under any sampling measure with full support).  However, this alone provides limited information on the probability of the event, and in particular does not identify situations where the probability is close to zero or close to one. Quantitative analysis across parameter space can therefore provide additional insight, and the particular methods used here will yield sufficiently explicit analytic results to identify regimes in which the occurrence of at least one zero of the Melnikov function is prevalent or rare.
\end{remark}

\begin{q}\label{q:main}
What is the probability that the Melnikov function is zero-free when the parameters $\delta, a_k, b_k$ in the differential equation \eqref{eq:DuffingSingle} are randomly sampled?
\end{q}

As it turns out, the Melnikov function $M_{N, \omega}$ in this case is itself a trigonometric polynomial (see Theorem \ref{thm:Melnikov}) given by
\be
M_{N, \omega}(t_0)=-\frac{4}{3}\delta+\sqrt{2}\pi\omega\sum_{k=1}^N \sech \left( \frac{\pi k \omega}{2} \right) \big[ka_k \sin (\omega k t_0) - kb_k \cos (\omega k t_0)  \big] .
\ee
Hence, when the forcing function is a random trigonometric polynomial with i.i.d. coefficients, Question \ref{q:main} reduces to studying the zero-free event for a random trigonometric polynomial having independent but \emph{not identically distributed} coefficients (the random coefficients are scaled by frequency-dependent weights $(\omega k) \sech \left( \frac{\pi \omega k}{2} \right)$).
We provide estimates addressing this question, while considering separately the cases of small and large period (or equivalently, large and small frequency $\omega$).
Our main results show that these two cases represent contrasting regimes.
Let $Z_{N,\omega}$ denote the number of zeros of $M_{N,\omega}$
over one period. We prove the following estimates (see Section \ref{sec:prob}):
\begin{itemize}
\item At high frequency, the Melnikov function is zero-free with
overwhelming probability. More precisely, we have 
\[
\mathbb{P}(Z_{N,\omega}\geq1)=O(e^{-\omega}),
\qquad \omega\to\infty,
\]
see Theorem \ref{thm:nozeros}.

\item At low frequency, the Melnikov function has zeros with
overwhelming probability, provided the degree grows sufficiently
rapidly. More precisely, for every $\gamma>0$, there exist
constants $c,C>0$ such that
\[
\mathbb{P}(Z_{N,\omega}=0)
\leq C\exp(-c\omega^{-1/4})
\]
for all sufficiently small $\omega>0$ and all $N$ satisfying
$N\omega>\gamma$, see Theorem \ref{thm:zero-free}.
\end{itemize}
Here, the phrase \emph{with overwhelming probability} means that the exceptional
probability decays faster than every power of $\omega^{-1}$
in the high-frequency regime, or of $\omega$ in the
low-frequency regime.

En route to establishing these results, we determine another
quantity of independent interest, the average number of zeros
of the Melnikov function, and analyze its asymptotics,
see Section \ref{sec:exp}.

The proofs combine tools from perturbation theory of dynamical systems and the study
of Gaussian random functions.
The expected number of zeros of the Melnikov function is determined using  the Kac--Rice formula, which yields the high-frequency estimate through Markov's inequality.
The low-frequency result
(Theorem \ref{thm:zero-free}) requires a different argument:
correlation estimates at suitably separated points allow us
control the zero-free probability through a Gaussian orthant estimate.

\begin{remark}
For each fixed $N\geq1$ and $\omega>0$, the Melnikov function
$M_{N,\omega}$ almost surely has only simple zeros.
Consequently, almost surely, each zero gives rise to a transverse
homoclinic intersection of the perturbed system for all
sufficiently small $\e>0$.
The almost sure simplicity of the zeros follows from an application of Bulinskaya's Lemma \cite[Prop. 2.3]{AzaisBul}; indeed, to apply this result, one simply observes that the Melnikov function has smooth samples (each sample is a trigonometric polynomial), and the Gaussian densities of its point evaluations are uniformly bounded since
its variance is constant and strictly positive.
\end{remark}

While the probabilistic perspective introduced here appears to be completely novel to this particular setting, it shares some of its formal structure with the recent approach used by the second-named author together with M. Krishnapur and O. Nguyen to study the zeros of Melnikov functions arising in a dynamically distinct setting of enumerating limit cycles in autonomous planar systems \cite{Lundberg}, \cite{KLN}.

%Holmes and Marsden \cite{HolmesMarsden} Marsden survey \cite{Marsden}

%Melnikov method for homoclinic connections \cite{GuckHolmes}, \cite{Rand}, 

%Melnikov method in other areas: limit cycles \cite[Sec. 2.1 of Part II]{ChLi} \cite[Sec. 26]{IlyashenkoBook}, celestial mechanics \cite{Peterson}, 

%In this paper we consider the Melnikov function mentioned in \cite{GuckHolmes} defined for time-periodic systems of the form
%$$\dot{x}=f(x)+\epsilon g(x,t);  \quad  x=\begin{pmatrix} u \\ v

%\end{pmatrix}\in \mathbb{R}^2$$

%where the Melnikov function is defined as 
%$$M(t_0)=\int_{-\infty}^{\infty} f(q^0(t)) \wedge g(q_0(t), t+t_0) dt$$ 
%where $M(t_0)$ provides a measure of the separation of the manifolds at $q^0(0)$ on $\Sigma^{t_0}$, where $q$

\begin{figure}[h]
\centering
\includegraphics[scale=0.2]{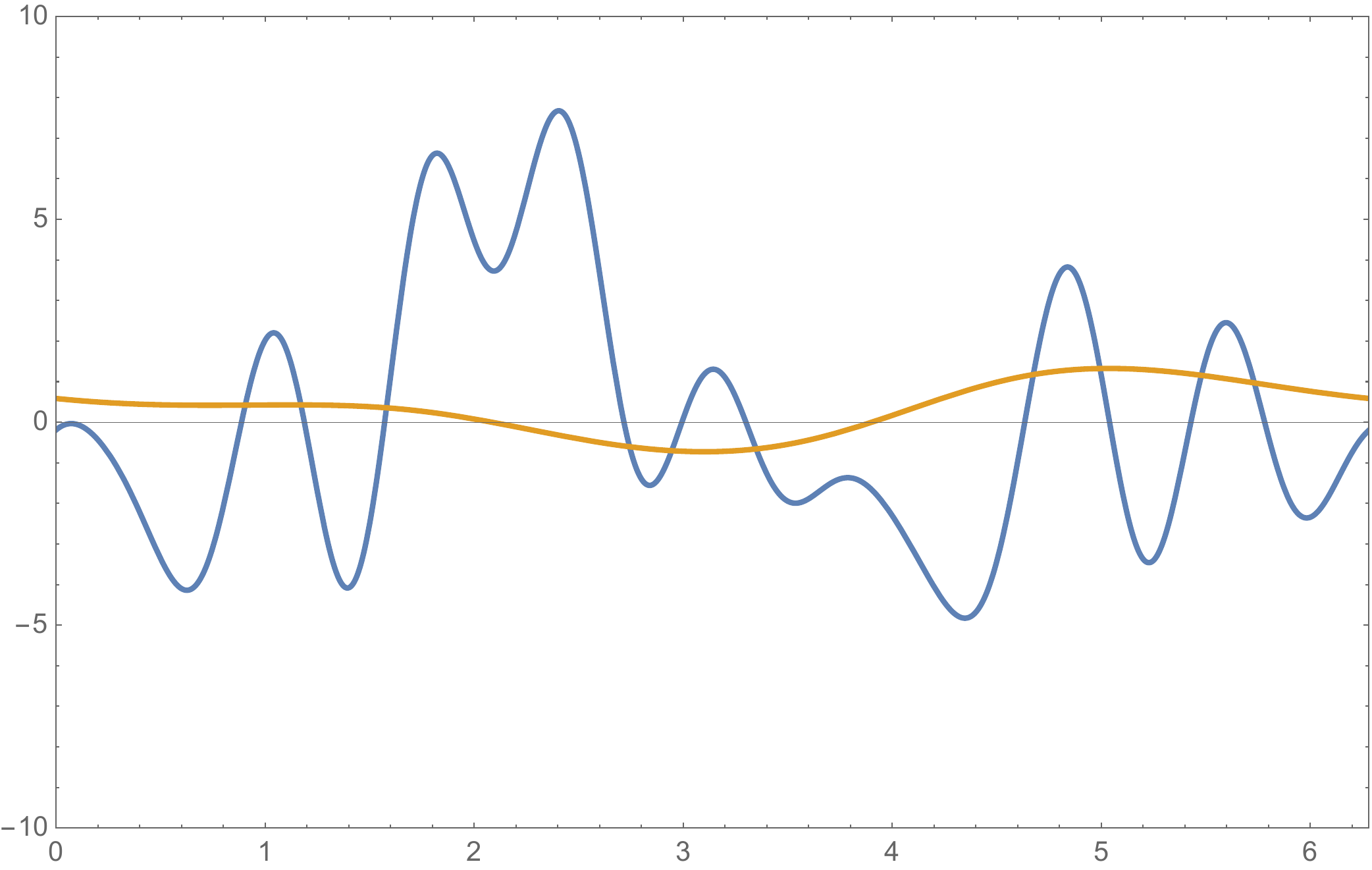} 
\hspace{.1in}
\includegraphics[scale=0.2]{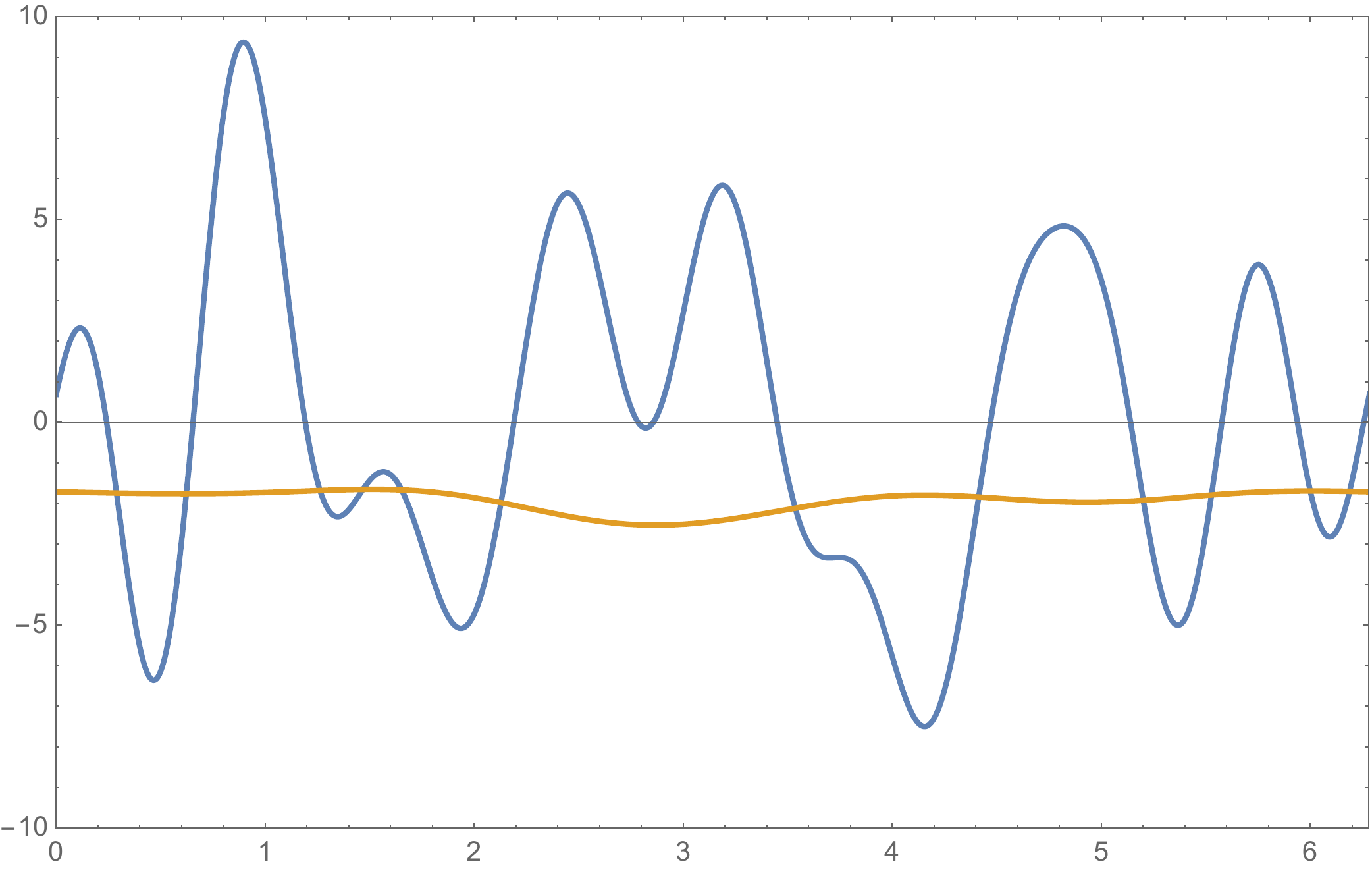} 
\caption{Random periodic forcing functions (trigonometric polynomials) of degree $10$ with their corresponding Melnikov functions overlaid.
In the sample on the left, the Melnikov function has two simple zeros, corresponding to transverse intersections of the stable and unstable manifolds (for small $\e>0$). In the sample on the right, the Melnikov function has no zeros, and hence, for sufficiently small $\e>0$, the stable and unstable manifolds do not intersect. Plots created using Mathematica.}
\label{fig:MelnikovSamples}
\end{figure}

%%Many different parameters, good to look at what is the expected outcome on average. 

While we focus on equation \eqref{eq:DuffingSingle} for concreteness, the methods used in this paper extend to other perturbed Hamiltonian systems having a homoclinic orbit for $\e = 0$, see the remarks in Section \ref{sec:concl}.

\subsection*{Outline of the paper}

In Section \ref{sec:prelim},
we review two preliminary results: 
the Melnikov distance for perturbed Hamiltonian systems and the Kac-Rice formula for the average number of zeros of random functions.
In Section \ref{sec:exp}, we apply those results to determine the average number of zeros for the Melnikov function associated to \eqref{eq:DuffingSingle}, where we also provide asymptotic estimates.
In Section \ref{sec:prob}, we provide asymptotic estimates for the probability of the event that the Melnikov function has no zeros.
We conclude with some discussion of future directions and open problems in Section \ref{sec:concl}. That section also gives an extension of the high-frequency result
(see Theorem \ref{thm:generalhighfrequency}) to a broad class of perturbed Hamiltonian systems.
The Appendix provides details for the computation of the Melnikov integral using contour integration and residues.

%\commE{ Maybe include a picture for this result with what system looks like with chaotic choice of delta and gamma}

\section{Preliminaries}\label{sec:prelim}

\subsection{Melnikov's Method}
We give a brief description of the Melnikov method following the presentation in \cite{GuckHolmes}, see Section 4.5 of that text for more details. Melnikov's method provides a criterion for detecting transverse homoclinic intersections in certain dynamical systems subject to small periodic perturbations. Consider a system of the form 
\be
\label{eq:hamilform} \dot{x}=h(x)+\epsilon g(x,t), \quad x=\begin{pmatrix} u \\v
\end{pmatrix}\in \mathbb{R}^2,
\ee
where $g(x,t)$ is a smooth function of period $T$ in $t$ and $h(x)$ is a Hamiltonian vector field, i.e.   
$h_1=\frac{\partial H }{\partial v}, h_2=-\frac{\partial H }{\partial u} $.  Further assume that the unperturbed system $\dot{x}=h(x)$ has a hyperbolic saddle point $p_0$ which is connected to itself via a homoclinic orbit $q^0(t)$.

We consider an autonomous version of the system in order to introduce the Melnikov integral: 
\begin{align*}
     \dot{x}&=h(x)+\varepsilon g(x, \theta)\\
    \dot{\theta}&=1 
\end{align*}
where $\theta$ is regarded modulo $T$. Thus, the extended phase space is $\mathbb{R}^2\times S^1 $. For each $t_0\in [0, T)$ we consider the cross section 
$$\Sigma^{t_0}=\{(x, \theta): \theta=t_0  \quad (\text{mod }T)\}$$ and denote the corresponding period-$T$ Poincar\'e map by $P_{\varepsilon}^{t_0}$. 

For $\varepsilon=0$ the saddle point $p_0$ corresponds to a hyperbolic fixed point of $P_0^{t_0}$, and the homoclinic orbit $q^0(t)$ lies in the coincident stable and unstable manifolds associated with this fixed point. Under a sufficiently small perturbation, $P_{\varepsilon}^{t_0}$ will still have a hyperbolic fixed point $p_\epsilon^{t_0}$, but the corresponding manifolds generally split. The Melnikov method measures the first-order separation between the perturbed stable and unstable manifolds along the unperturbed homoclinic orbit.

The Melnikov function is defined by 
\be\label{eq:MelnikovGeneral}
M(t_0)=\int_{-\infty}^{\infty} \nabla H(q^0(t)) \cdot g(q^0(t), t+t_0) dt,
\ee
which can be interpreted as measuring the accumulated first-order transverse effect of the perturbation along the unperturbed homoclinic orbit. More precisely, $M(t_0)$ determines the leading-order signed separation between the perturbed stable and unstable manifolds. 

Suppose that $M$ possesses a simple zero at $t_0=t_0^*$, i.e., 
$$M(t_0^*)=0\quad M'(t_0^*)\neq 0$$
Then for sufficiently small nonzero $\varepsilon$, the stable and unstable manifolds of the saddle point $p_\epsilon^{t_0}$ intersect transversely near the unperturbed homoclinic orbit $q^0(t)$ \cite[\text{Theorem} 4.5.3]{GuckHolmes}. Consequently, the perturbed Poincar\'e map possesses a transverse homoclinic point. 

The Smale-Birkhoff Homoclinic Theorem \cite[\text{Theorem} 5.3.5]{GuckHolmes} then implies that some iterate of the Poincaré map contains a hyperbolic invariant set on which the dynamics are conjugate to a shift on a finite set of symbols. Hence, the existence of a simple zero of the Melnikov function provides a sufficient condition for horseshoe-type chaotic dynamics in the perturbed system.

We also note that when $M$ is zero-free the stable and unstable manifolds do not intersect for $\e>0$ sufficiently small \cite[\text{Theorem} 4.5.3]{GuckHolmes}.

Let us apply this to the equation \eqref{eq:DuffingSingle}
while taking $u = y, v = \dot{y}$  to write it in the form \eqref{eq:hamilform}:
\begin{equation}\label{eq:Duffing}
\begin{bmatrix}
\dot{u}\\
\dot{v}
\end{bmatrix}=\begin{bmatrix}
v\\
u - u^3
\end{bmatrix} + \e\begin{bmatrix}
0\\
f(t) - \delta v
\end{bmatrix}=h(u,v)+\e g(u,v;t).
\end{equation}

\begin{thm}\label{thm:Melnikov}
If $f(t) = f_{N,\omega}(t)$ is a trigonometric polynomial of the form 
\begin{equation}\label{eq:fTrigPoly}
f_{N, \omega}(t)=\sum_{k=1}^{N} a_k \cos (\omega k t)+b_k \sin(\omega k t)
\end{equation}
then the Melnikov function associated to \eqref{eq:Duffing} is given by 
\begin{equation}\label{eq:MelnikovTrig}
M_{N, \omega}(t_0)=-\frac{4}{3}\delta+\sqrt{2}\pi\omega\sum_{k=1}^N \sech \left( \frac{\pi k \omega}{2} \right) \big[ka_k \sin (\omega k t_0) - kb_k \cos (\omega k t_0)  \big] .
\end{equation}
\end{thm}

The essential underlying property of the Melnikov function used here is that it is linear in the perturbation.

\begin{proof}  
When $\e=0$, the system \eqref{eq:Duffing} has a hyperbolic saddle at $(0,0)$ and a pair of homoclinic orbits to $(0,0)$ based at $(\pm \sqrt{2}, 0)$.
Restricting our attention to the right halfplane, the orbit with initial condition $q^0(0)=(\sqrt{2},0)$ is given by
\be
q^0(t) = (\sqrt{2} \sech  t, -\sqrt{2} \sech  t \tanh t).
\ee
Then our goal is to compute the Melnikov integral associated to this homoclinic orbit 
\be\label{eq:MelnikovGoal}
M(t_0) = \int_{-\infty}^{\infty} \nabla H(q^0(t)) \cdot g(q^0(t), t+t_0)\, dt.
\ee
We find
\[
\nabla H(q^0(t)) =
\begin{bmatrix}
-\sqrt{2}\,\sech(t) + 2\sqrt{2}\,\sech^3(t)\\[4pt]
-\sqrt{2}\,\sech(t)\tanh(t)
\end{bmatrix},\]
and
\[
g(q^0(t), t+t_0) =
\begin{bmatrix}
0\\[4pt]
f_{N,\omega}(t+t_0) + \delta\sqrt{2}\,\sech(t)\tanh(t)
\end{bmatrix},
\]
and hence the integrand of the Melnikov function \eqref{eq:MelnikovGoal} is given by
\[
\nabla H(q^0(t)) \cdot g(q^0(t), t+t_0)
= -\bigg(
\sqrt{2}\,\sech(t)\tanh(t)f_{N, \omega}(t+t_0)
+ 2\delta\,\sech^2(t)\tanh^2(t)
\bigg).
\]

Therefore,
\be\label{eq:penultimate}
M(t_0)
= - 2\delta\int_{-\infty}^{\infty} \sech^2(t)\tanh^2(t)\, dt -\sqrt{2}\int_{-\infty}^{\infty} \sech(t)\tanh(t)f_{N, \omega}(t+t_0)\, dt.
\ee
The first integral is easily evaluated using the substitution $u = \tanh(t)$, which gives
\begin{equation}\label{eq:usub}
-2\delta\int_{-\infty}^{\infty} \sech^2(t)\tanh^2(t)\, dt
= -2\delta \int_{-1}^{1} u^2\, du
= -\frac{4\delta}{3}.
\end{equation}

Recalling the expression \eqref{eq:fTrigPoly} for $f_{N,\omega}$ and using linearity of the integral to write the second integral in \eqref{eq:penultimate} as a sum of integrals, the desired result \eqref{eq:MelnikovTrig} is then a consequence of the following pair of identities (applying each identity $N$ times while taking $\nu = \omega k$, $k =1,2,\dots,N$)
\begin{equation}\label{eq:residue1}
\int_{-\infty}^{\infty} \sech(t)\tanh(t) \cos(\nu(t+t_0))\, dt
= 
-\pi \nu \sin(\nu t_0)
\sech\!\left(\frac{\pi\nu}{2}\right),
\end{equation}
\begin{equation}\label{eq:residue2}
\int_{-\infty}^{\infty} \sech(t)\tanh(t) \sin(\nu(t+t_0))\, dt
= 
 \pi \nu \cos(\nu  t_0)
\sech\!\left(\frac{\pi\nu}{2}\right).
\end{equation}
A detailed derivation of \eqref{eq:residue1} and \eqref{eq:residue2} using residues is provided in the Appendix.
\end{proof}

\subsection{The Kac-Rice formula}

We will need the Kac-Rice formula, a fundamental result from the study of zeros of random functions.
There are versions of the Kac-Rice formula that hold in more general settings, but we will just need the following univariate Gaussian formulation from \cite[Thm. 3.2]{AzaisWscheborbook}.

\begin{thm}[Kac-Rice formula]\label{thm:KR}
Let $f:I \rightarrow \RR$ be a random function with $I \subset \RR$ an interval.
Suppose the following conditions are satisfied
\begin{itemize}
    \item[(i)] $f$ is Gaussian,
    \item[(ii)] $f$ is almost surely of class $C^1$,
    \item[(iii)] For each $t \in I$, the random vector $(f(t),f'(t))$ has a nondegenerate distribution,
    \item[(iv)] Almost surely $f$ has no degenerate zeros.
\end{itemize}
Then 
\be\label{eq:KRformula}
\EE |\{t \in I:f(t) =0\} | = \int_I \int_\RR |B| \rho_t(0,B) dB dt,
\ee
where $ \rho_t(A,B)$ denotes the
joint probability density function
of $A=f(t)$ and $B=f'(t)$.
\end{thm}

Furthermore, if the process is centered, we can express \eqref{eq:KRformula} in terms of the covariance kernel $K(r,t):=\EE f(r) f(t)$, namely \cite{EdelmanKostlan95},
\be\label{eq:EKform}
\EE |\{t \in I:f(t) =0\} | = \frac{1}{\pi} \int_I \sqrt{  \left.  \frac{\partial^2}{\partial t \partial r} \log K(r,t) \right|_{r=t=\tau}} d\tau.
\ee
This identity is based on the fact that the
joint density $\rho_t(A,B)$ of the Gaussian pair $A=f(t), B=f'(t)$  can be expressed in terms of the covariance matrix of $f(t)$ and $f'(t)$ which in turn can be computed from evaluation along the diagonal $r=t$ of appropriate partial derivatives of $K(r,t)$.

Applying \eqref{eq:EKform} to the case of random trigonometric polynomials yields the following result.

\begin{thm}[Random trigonometric polynomials with weights]\label{thm:RandTrigWeights}
Let
\begin{equation}\label{eq:TrigPoly}
T_n(t)
= \sum_{k=0}^{n} \sigma_{k}\, a_k \cos(\omega k t)
  + \sum_{k=1}^{n} \sigma_{k}\, b_k \sin(\omega k t), 
\end{equation}
where $a_k, b_k$ are i.i.d.\ $N(0,1)$ random variables,
and $\sigma_k$ is a deterministic sequence of real weights, not all zero.

Let $N_{[a,b]}(T_n)$ denote the number of real zeros of $T_n$ in the interval $[a,b]$.
Then
\begin{equation}\label{eq:KacRice}
\mathbb{E}\,N_{[a,b]}(T_n)
= \frac{b-a}{\pi}\,
\sqrt{\frac{\sum_{k=1}^n (\omega k)^2\,\sigma_{k}^2}{\sum_{k=0}^n \sigma_{k}^2}}.
\end{equation}

\end{thm}

\section{Results on the expected number of zeros}\label{sec:exp}

%\commE{Add statements of our first two results along with brief description of the proofs.}

%For the sake of what is to come, we will rewrite the Melnikov function \eqref{eq:MelnikovTrig} as 

Since the zero set of the Melnikov function \eqref{eq:MelnikovTrig} is invariant under scaling, we will rewrite it as
\begin{equation}\label{eq:Melnikov_N}
M_{N, \omega}(t_0)=\frac{2\sqrt{2}}{3\pi}\delta+\sum_{k=1}^N \omega k \sech \left(\frac{\pi \omega k}{2} \right) \big[a_k \sin (\omega k t_0)+ b_k \cos (\omega k t_0)  \big],
\end{equation}
where we have also absorbed the negative signs into $\delta, b_k$ using the symmetry of centered Gaussians.
We let
$$Z_{N, \omega}:= |\{t_0\in[0,T):M_{N,\omega}(t_0)=0\}|$$
denote the number of zeros of the Melnikov function $M_{N, \omega}(t_0)$ in a single period.

The following result provides an explicit formula for the average number of zeros of the Melnikov function. This will follow as a direct application of Theorems \ref{thm:Melnikov} and \ref{thm:KR}.

\begin{thm}\label{thm:expectation}
Consider the function \eqref{eq:Melnikov_N}. Let $a_k, b_k, {\delta}$ be independent and identically distributed (i.i.d) standard normal random variables. 
The expected number of zeros of the Melnikov function $M_{N, \omega}(t_0)$ on the interval $[0, \frac{2\pi}{\omega})$ is given by 
\be\label{eq:expectationN}
\EE \, Z_{N, \omega} = \frac{2}{\omega}\cdot \sqrt{\frac{\sum_{k=1}^{N} (\omega  k)^4 \sech^2(\frac{k\pi \omega}{2})}{8/(9\pi^2)+\sum_{k=1}^{N} (\omega k)^2\sech^2(\frac{k\pi \omega}{2})}}.
\ee
\end{thm}

\begin{proof}
This follows immediately from applying the Kac-Rice formula stated in Theorem \ref{thm:RandTrigWeights} to the Melnikov function provided in \eqref{eq:Melnikov_N}
where we note that when the parameters in the perturbed system \eqref{eq:Duffing} are i.i.d. random Gaussians, the Melnikov function $M_{N,\omega}$ is indeed a random trigonometric polynomial of the form \eqref{eq:TrigPoly} with weights $\sigma_k=\omega k\sech(\frac{k\pi\omega}{2})$ for $k\geq 1$ and $\sigma_0=2\sqrt{2}/(3\pi)$.
\end{proof}

It follows from Theorem \ref{thm:expectation} that $\EE \,Z_{N,\omega}$ decays exponentially as $\omega \rightarrow \infty$, see the proof of Theorem \ref{thm:nozeros} below for a detailed estimate that will be used to further estimate the probability that $Z_{N,\omega}=0$, i.e., $M_{N,\omega}$ does not have any zeros.

\begin{cor} 
As $N\rightarrow \infty$, the expected number of zeros of \eqref{eq:MelnikovTrig} approaches a finite constant value $C_\omega$ (depending on $\omega$) given by 
\begin{equation}
C_\omega=\frac{2}{\omega}\cdot \sqrt{\frac{\sum_{k=1}^{\infty} (\omega  k)^4 \sech^2(\frac{k\pi \omega}{2})}{8/(9\pi^2)+\sum_{k=1}^{\infty} (\omega k)^2\sech^2(\frac{k\pi \omega}{2})}}.
\end{equation}
\end{cor}

%For the expected number of zeros result we will cite \cite{EdelmanKostlan95}.
\begin{proof}
Using the basic estimate $\sech^2(\frac{k\pi\omega}{2})\leq 4e^{-k\pi\omega}$, the terms of the series appearing in the numerator and denominator of \eqref{eq:expectationN} are bounded by $4(\omega k)^4 e^{-\pi k \omega}$ and $4(\omega k)^2 e^{-\pi k \omega}$ respectively. Since each of these sequences is summable, this shows absolute convergence of each of the series appearing in the numerator and denominator. 
\end{proof}

\begin{thm}
The limiting expectation $C_\omega$ satisfies
\begin{equation}
C_\omega \sim \sqrt{\frac{7}{5}} \cdot {\frac{2}{\omega}}, \quad (\omega \rightarrow 0).
\end{equation}

%%[statement about average number of zeros being arbitrarily large for appropriate choice of small $\omega>0$ and large $N$--perhaps this can be stated as there exists $\omega_0>0$ and $N_0>0$ such that for all $0<\omega<\omega_0$ and $N>N_0$--first choose $\omega_0$ and then see if we have uniform (in $\omega$) convergence to $C_\omega$ as $N \rightarrow \infty$?]

\end{thm}

\begin{proof}

First rewrite the sum
$$\sum_{k=1}^{\infty} (\omega k)^4\sech^2\bigg(\frac{k\pi\omega}{2}\bigg)\omega $$
as a sum of integrals: 
\be\label{eq:sumtointegral}
\sum_{k=1}^{\infty} (\omega k)^4\sech^2\bigg(\frac{k\pi\omega}{2}\bigg)\omega= \sum_{k=1}^{\infty} \int_0^\infty(\omega k)^4\sech^2\bigg(\frac{k\pi\omega}{2}\bigg) \chi_{[\omega(k-1), \omega k]}(x) dx
\ee
Since these integrands are all nonnegative, a standard corollary \cite[Ch. 2, Cor. 1.10]{Stein} of the monotone convergence theorem  gives 
\be\label{eq:MCT}
\sum_{k=1}^{\infty} (\omega k)^4\sech^2\bigg(\frac{k\pi\omega}{2}\bigg)\omega= \int_0^\infty\sum_{k=1}^{\infty}(\omega k)^4\sech^2\bigg(\frac{k\pi\omega}{2}\bigg) \chi_{[\omega(k-1), \omega k]}(x) dx.
\ee

The integrand in the line above is the approximation of the function $f(x) = x^4 \sech^2(\frac{\pi x}{2})$ via step functions. The function $f$ is continuous, bounded by a constant $C$, and monotone decreasing on  $(x^*,\infty)$ for some $x^*>0$. Since the approximation by step functions uses right endpoints, it is dominated by $f$ on $(x^*,\infty)$ and hence is dominated on $[0,\infty)$ by the integrable function $C \chi_{[0,x^*]} + f(x)\chi_{(x^*, \infty)}$.

%$$ g(x) := \begin{cases}
%C, \qquad 0\leq x\leq  x^*\\
%f(x), \qquad x > x^*
%\end{cases}$$

%We then note that the function 
%$$x^4 \sech^2(\frac{\pi x}{2})$$ 
%achieves a maximum at the value $x_\text{max}$ where $4-\pi x \tanh(\frac{\pi}{2}x_{\text{max}})=0$.  We then consider the integral from 0 to $x_\text{max}$ and from $x_\text{max}$ to $\infty$. From 0 to $x_\text{max}$, since $x^4 \sech^2(\frac{\pi x}{2})$ is bounded, we can swap the limit in. For the case of the integral from $x_\text{max}$ to $\infty$, we have that
%$$\bigg|\sum_{k=1}^{\infty}(\omega k)^4\sech^2(\frac{k\pi\omega}{2})\chi_{[\omega (k-1), \omega k]}(x)\bigg|\leq  x^4\sech^2(\frac{\pi x}{2}).$$
%Since $g(x)$ is $L_1$, 
%\commE{fix the step function argument, left side bounded by constant, right side by shift of function. }
%From $x=0$ to this maximum, the step function $\sum_{k=1}^{\infty}(\omega k)^4\sech^2(\frac{k\pi\omega}{2})\chi_{[\omega k, \omega(k+1)]}(x) d\mu$ will be an underestimate and for $x>x_{\text{max}}$, the step function is an overestimate. Then we can split the sum into the pieces before and after $x_{\text{max}}$. For the part after $x_{\text{max}}$, we can dominate the step function by $x^4\sech^2(\frac{\pi}{2}x)$ which is integrable on $\mathbb{R}$. Then by dominated convergence theorem we get: 
%$$\int_0^\infty\sum_{k=1}^{\infty}(\omega k)^4\sech^2\bigg(\frac{k\pi\omega}{2}\bigg) \chi_{[\omega(k-1), \omega k]}(x) d\mu\rightarrow \int_0^\infty x^4\sech^2(\frac{\pi x}{2})\chi_{[0, \infty)}(x)d\mu=\int_{0}^{\infty} x^4\sech^2(\frac{\pi x}{2})\  dx$$
%as $\omega\rightarrow 0$. 

We then have, by Lebesgue's dominated convergence theorem along with \eqref{eq:MCT},
\be\label{eq:numerator}
\lim_{\omega \rightarrow 0} \sum_{k=1}^{\infty} (\omega k)^4\sech^2\bigg(\frac{k\pi\omega}{2}\bigg)\omega =\int_{0}^\infty x^4\sech^2(\frac{\pi}{2}x)\ dx.\ee

A similar argument can be used to show that \be
\label{eq:denominator}
\sum_{k=1}^{\infty} (\omega k)^2\sech^2(\frac{k\omega\pi}{2})\omega\rightarrow \int_{0}^\infty x^2\sech^2(\frac{\pi}{2}x)\ dx
\ee 
as $\omega\rightarrow 0$.

Now we use \eqref{eq:numerator} and \eqref{eq:denominator} to establish the statement in the theorem, namely, to show 
$$\displaystyle \lim_{\omega \rightarrow 0} \frac{\omega}{2} C_\omega = \sqrt{\frac{7}{5}}.$$ 
We have
\begin{align*}
\lim_{\omega\rightarrow 0} \frac{\omega}{2} \cdot C_{\omega}&=\lim_{\omega\rightarrow 0} \sqrt{\frac{\sum_{k=1}^{\infty} (\omega  k)^4 \sech^2(\frac{k\pi \omega}{2})\omega}{8\omega/(9\pi^2)+\sum_{k=1}^{\infty} (\omega k)^2\sech^2(\frac{k\pi \omega}{2})\omega}}\\
&=\sqrt{\frac{\int_{0}^\infty x^4\sech^2(\frac{\pi}{2}x)\ dx}{\int_{0}^\infty x^2\sech^2(\frac{\pi}{2}x)\ dx}}\\
&= \sqrt{\frac{\frac{14}{15\pi}}{\frac{2}{3\pi}}} =\sqrt{\frac{7}{5}}
\end{align*}
where the integrals are evaluated using the table of integrals \cite[Sec. 3.527, Item 5]{Grad} which gives, for $a>0$ and any integer $m\geq1$,
$$\int_{0}^{\infty} \frac{x^{2m}}{\cosh^2(ax)}\ dx=\frac{(2^{2m}-2)\pi^{2m}}{(2a)^{2m}a}\cdot |B_{2m}|$$
where $B_{2m}$ is the associated Bernoulli number. 
\end{proof}

% \newpage
% \textbf{Other Way?}

% \begin{align*}
% \bigg|\sum_{k=1}^{\infty} (\omega k) ^2 \sech^2 (\frac{k\pi \omega}{2}) \omega - \int_{0}^ \infty
%  x^2 \sech^2( \frac{\pi}{2}x) \ dx\bigg|&\leq \bigg| \sum_{k=1}^{\infty} (\omega k) ^2 \sech^2 (\frac{k\pi \omega}{2}) \omega-\sum_{k=1}^{N} (\omega k) ^2 \sech^2 (\frac{k\pi \omega}{2}) \omega\bigg|
%  \\&+\bigg|\sum_{k=1}^{N} (\omega k) ^2 \sech^2 (\frac{k\pi \omega}{2}) \omega - \int_{0}^ {\omega N} x^2 \sech^2( \frac{\pi}{2}x) \ dx   \bigg|
%  \\&+ \bigg|\int_{0}^ {\omega N} x^2 \sech^2( \frac{\pi}{2}x) \ dx-\int_{0}^ {\infty} x^2 \sech^2( \frac{\pi}{2}x) \ dx     \bigg|
% \end{align*}
% Choose $\omega<<0$ so that the third difference is less than $\epsilon/3$.
% \\\\
% \textbf{Why is this possible:} 
% \\
% Rewrite this difference as: 
% \begin{align*}
% \bigg|\sum_{k=1}^{N} (\omega k) ^2 \sech^2 (\frac{k\pi \omega}{2}) \omega - \int_{0}^ {\omega N} x^2 \sech^2( \frac{\pi}{2}x) \ dx   \bigg|&=\bigg|\sum_{k=1}^{N} (\omega k) ^2 \sech^2 (\frac{k\pi \omega}{2}) \omega - \int_{0}^ {\omega N} x^2 \sech^2( \frac{\pi}{2}x) \ dx   \bigg|
% \end{align*}

\section{Results on the probability of the zero-free event}\label{sec:prob}

\subsection{Low probability of zeros at high frequency}

\begin{thm}[Overwhelming probability of no zeros at high frequency]
\label{thm:nozeros}
Let $\ZNomega$ denote the number of zeros of $M_{N, \omega}(t)$ over one period.
We have, uniformly in $N\geq $, 
$$
\mathbb{P}\bigl( Z_{N,\omega} \ge 1\bigr) = O(e^{-\omega}), \quad \text{as } \omega \rightarrow \infty.
$$
\end{thm}

\begin{proof}
The proof will use Theorem \ref{thm:expectation} and Markov's inequality.

Recall  
\begin{align*}
\EE \,\ZNomega &=\frac{2}{\omega}\cdot \sqrt{\frac{\sum_{k=1}^{N} (\omega  k)^4 \sech^2(\frac{k\pi \omega}{2})}{8/(9\pi^2)+\sum_{k=1}^{N} (\omega k)^2\sech^2(\frac{k\pi \omega}{2})}}.
\end{align*}
Using the basic estimate $\displaystyle \sech^2\left(\frac{k\pi\omega}{2}\right)\leq 4e^{-k\pi\omega}$ in the numerator and bounding the denominator from below by $8/(9\pi^2)$, we obtain 
\begin{equation}\label{eq:expbound}
\EE \, \ZNomega \leq 3\pi \sqrt{2} \omega \sqrt{\sum_{k=1}^{\infty}  k^4(e^{-\pi\omega} )^k}.
\end{equation}

%\begin{align*}
%C_{N, \omega}&\leq \frac{2}{\omega} \sqrt{\sum_{k=1}^{N} (\omega k)^4 \sech^{2}(\frac{k\pi\omega}{2})}\\
%&\leq\frac{4}{\omega} \sqrt{\sum_{k=1}^{\infty} (\omega k)^4e^{-k\pi\omega} }\\
%&=4\omega \sqrt{\sum_{k=1}^{\infty}  k^4(e^{-\pi\omega} })^k\\
%\end{align*}
For the sake of notation, define $\xi:=e^{-\omega\pi}$. 
Using the elementary identity
$$\sum_{k=1}^{\infty}  k^4\xi^k= \left(\xi \frac{\partial}{\partial \xi} \right)^4 \frac{\xi}{1-\xi} =  \frac{\xi(1+11\xi+11\xi^2+\xi^3)}{(1-\xi)^5}$$
along with \eqref{eq:expbound}, we obtain
\begin{align*}
\EE \, \ZNomega&\leq3\pi \sqrt{2} \omega\sqrt{\frac{\xi(1+11\xi+11\xi^2+\xi^3)}{(1-\xi)^5}}\leq C\omega\sqrt{\xi}=C\omega e^{-\omega\pi/2}.
\end{align*}

Applying Markov's inequality then gives the result. 
$$\mathbb{P}\bigl(\ZNomega \ge 1\bigr)\leq \EE \, \ZNomega\leq C\omega e^{-\omega\pi/2} \leq C e^{-\omega}.$$
\end{proof}

\subsection{High probability of at least one zero for low frequency and high degree}

\begin{thm}[Overwhelming probability of zeros at low frequency] \label{thm:zero-free} Let $\gamma>0$. There exist constants $\omega_0>0$, $c>0$, and $C>0$ such that, whenever \[ 0<\omega<\omega_0 \qquad\text{and}\qquad N\omega>\gamma, \] we have \[ \PP\!\left( Z_{N, \omega} = 0 \right) \leq C\exp\!\left(-c\omega^{-1/4}\right). \] \end{thm}

\begin{remark}
The condition $N\omega>\gamma$ requires the degree $N$ to grow
as $\omega\to0$. Thus, Theorem \ref{thm:zero-free}
concerns a joint low-frequency and large-degree limit. In contrast, the high-frequency conclusion of
Theorem \ref{thm:nozeros} holds for fixed degree $N$.
\end{remark}

In order to estimate the probability of the Melnikov function having no zeros, we evaluate the function at a collection of sufficiently spaced points over one period. The idea is that since a zero-free continuous function must have a constant sign, the event in which a function has no zeros is contained in the event in which all these evaluations are either all positive or all negative. To then estimate the probability of this event, we introduce the covariance kernel and correlation function of $M_{N, \omega}(t_0)$. 

The covariance kernel $K(s,t)$ of $M_{N, \omega}$ is given by
\begin{align*}
K(s,t)
&:= \mathbb{E}[M_{N, \omega}(s)M_{N, \omega}(t)]\\[4pt]
&= \frac{8}{9\pi^2}
+ \sum_{k=1}^{N} 
  (k \omega)^2 \,\mathrm{sech}^2\!\big(\tfrac{\pi k \omega}{2}\big) \cos(\omega k (s-t)),
\end{align*}
which depends only on $s-t$ (this reflects that our process is stationary). Normalizing, we obtain the correlation function
\begin{align*}
\Corr(s-t)&= \frac{K(s,t)}{\sqrt{K(s,s)\cdot K(t,t)}}\\ 
&= \frac{8/(9\pi^2)+\sum_{k=1}^N(k \omega)^2 \,\mathrm{sech}^2\!\big(\tfrac{\pi k \omega}{2}\big)\cos(\omega k(s-t))}{8/(9\pi^2)+\sum_{k=1}^N (k \omega)^2 \,\mathrm{sech}^2\!\big(\tfrac{\pi k \omega}{2}\big) }.
\end{align*}

\begin{figure}[ht]
\centering
\includegraphics[scale=0.4]{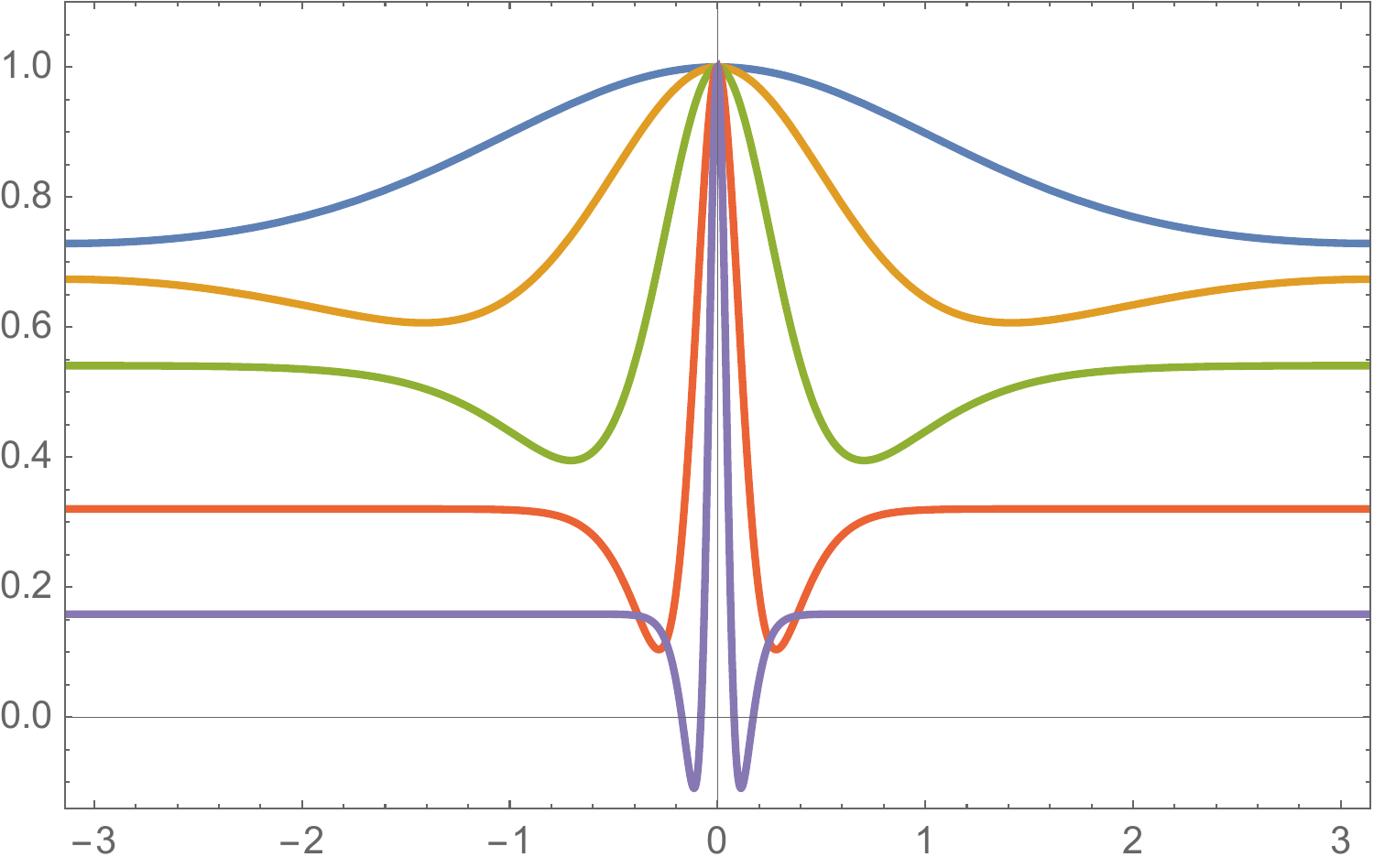}\\
\caption{Correlation function plots for $\omega=1,\ 0.5,\ 0.25,\ 0.1, \text{ and }0.04$ with $N=1000$.  Each plot shows a single period with the horizontal coordinate scaled to normalize the period to $2\pi$, i.e., the function plotted is $\Corr(\tau / \omega)$ for $\tau$ from $0$ to $2\pi$.}
\label{fig:corr}
\end{figure}

%The figure shows the correlation $\Corr(s,t)$ computed for $N=1000$ and different values of $\omega$, along with $s-t$ ranging from $0$ to $\frac{2\pi}{\omega}$. 

The following lemma provides the key estimate on the correlation which is seen to decay for sufficiently spaced points as $\omega \rightarrow 0$ (see also Figure \ref{fig:corr}).

\begin{lemma}[Correlation in the bulk]\label{lemma:corrinbulk}
Let $\gamma>0$.  There exist constants $C>0$, $\omega_0 > 0$ such that for 
$\tau\in\mathbb R\setminus\frac{2\pi}{\omega}\mathbb Z$, we have for $0<\omega<\omega_0$ and $N \omega > \gamma > 0$
\begin{equation}
|\Corr(\tau)| \leq \frac{C}{|\sin(\omega \tau /2)|} \omega \log (1/\omega).
\end{equation}
\end{lemma}

As part of the proof we require an elementary result using summation by parts.

\begin{lemma}[Summation by parts estimate for oscillatory sums]\label{lemma:sumparts}
Let $F \in C^1([0,\infty))$ with $\|F'\|_{\infty} \le M$, and $F(0)=0$.  Then for $\omega>0$ and $\tau\in\mathbb R\setminus\frac{2\pi}{\omega}\mathbb Z$
\[
\left| \sum_{k=0}^{N_0} F(k\omega)\cos(k \tau\omega) \right|
\le \frac{2}{|\sin(\tau\omega /2)|} \left( MN_0\omega  \right).
\]
\end{lemma}

\begin{proof}[Proof of Lemma \ref{lemma:sumparts}]
Set $a_k = F(k\omega)$ and $b_k = \cos(k\omega\tau)$. Define
\begin{equation}
\label{eq:Bk}
B_k := \sum_{j=0}^k b_j = \sum_{j=0}^k \cos(j\tau\omega).
\end{equation}

Recall the elementary trigonometric identity  related to the Dirichlet kernel \cite[Sec. 15.2]{bruckner1997real}:
\begin{equation}
\sum_{j=0}^k \cos(j \theta) = \frac{\sin((k+\frac{1}{2})\theta)+ \sin(\theta/2)}{2\sin(\theta/2)}.
\end{equation}

Applying this to \eqref{eq:Bk} and using the triangle inequality, we get the bound
\begin{equation}\label{eq:bound}
|B_k| \le \frac{1}{|\sin(\tau\omega/2)|}, \quad \tau\in\mathbb R\setminus\frac{2\pi}{\omega}\mathbb Z.
\end{equation}
Using summation by parts \cite[Thm. 3.41]{BabyRudin}, which states
\[
\sum_{k=0}^{N_0} a_k b_k
= a_{N_0} B_{N_0} - \sum_{k=0}^{{N_0}-1} (a_{k+1}-a_k)\, B_k,
\]
we obtain
\be\label{eq:SumPartsBound}
\left| \sum_{k=0}^{N_0} a_k b_k \right|
\le \frac{1}{|\sin(\tau\omega/2)|} \left( |a_{N_0}| + \sum_{k=0}^{{N_0}-1} |a_{k+1}-a_k| \right).
\ee

By the mean value theorem, we have
\[
|a_{k+1}-a_k| = |F((k+1)\omega) - F(k\omega)| \le M \omega,
\]
and
\[
|a_{N_0}| = |F(N_0\omega)| \le M N_0 \omega.
\]
Applying these estimates to \eqref{eq:SumPartsBound}, we obtain
\[
\left| \sum_{k=0}^{N_0} a_k b_k \right|
\le \frac{1}{|\sin(\tau\omega/2)|} 2 M N_0 \omega,
\]
as desired.
\end{proof}

\begin{proof}[Proof of Lemma \ref{lemma:corrinbulk}]
Recall
\be\label{eq:Corr}
\Corr(\tau)= \frac{8/(9\pi^2)+\sum_{k=1}^N (k \omega)^2 \,\mathrm{sech}^2\!\big(\tfrac{\pi k \omega}{2}\big) \cos(\omega k\tau)}{8/(9\pi^2)+\sum_{k=1}^N (k \omega)^2 \,\mathrm{sech}^2\!\big(\tfrac{\pi k \omega}{2}\big) },
\ee
which we write as (multiplying the numerator and denominator by $\omega$ and replacing $8/(9\pi^2)$ with $\alpha_0$)
\be\label{eq:CorrRiemann}
\Corr(\tau)= \frac{\omega\alpha_0+\sum_{k=1}^N (k \omega)^2  \,\mathrm{sech}^2\!\big(\tfrac{\pi k \omega}{2}\big) \cos(\omega k\tau) \omega }{\omega\alpha_0+\sum_{k=1}^N (k \omega)^2 \,\mathrm{sech}^2\!\big(\tfrac{\pi k \omega}{2}\big) \omega}.
\ee

Denote $F(x)=x^2\sech^2(\pi x/2)$. 
First we estimate the denominator of the correlation function using the fact that $F(x)$
is strictly increasing on some interval providing a simple estimate for the denominator of \eqref{eq:Corr} viewed as a Riemann sum.  Namely, $F(x)$ is increasing on an interval
$(0,\gamma^*)$ for some $0<\gamma^* \leq \gamma$.  Take $\omega_0 = \gamma^*/2$
so that for $0<\omega < \omega_0$ there exists $N^*$ such that $\omega_0 = \gamma^*/2 < N^* \omega < \gamma^*$.  Then 
we have
\begin{align*}
\sum_{k=1}^{N} F(k\omega)\omega \geq \sum_{k=1}^{N^*} F(k\omega)\omega >\int_{0}^{\omega_0} F(x)\ dx>0,
\end{align*}
since the second sum gives a Riemann sum with right end points which is an overestimate for the integral since our integrand strictly increases over the interval $(0, \gamma^*)$.

Let
\[
N_0=\min\left\{N,\left\lceil\omega^{-1}\log(1/\omega)\right\rceil\right\},
\qquad
C_0=\left(\int_0^{\omega_0}F(x)\,dx\right)^{-1}.
\]
Splitting the sum at $N_0$, we obtain
\begin{equation}\label{eq:corr-split}
|\Corr(\tau)|
\leq C_0\Bigg(
\omega\alpha_0
+\omega\left|\sum_{k=1}^{N_0}F(k\omega)\cos(k\omega\tau)\right|
+\omega\left|\sum_{k=N_0+1}^{N}F(k\omega)\cos(k\omega\tau)\right|
\Bigg),
\end{equation}
where the tail sum is empty when $N_0=N$.

To bound the first sum, we apply Lemma \ref{lemma:sumparts}
with $F$ as above and $M$ a bound on its derivative. Since
$N_0\leq\omega^{-1}\log(1/\omega)+1$, this gives
\begin{equation}\label{eq:first-sum}
\omega\left|\sum_{k=1}^{N_0}F(k\omega)\cos(k\omega\tau)\right|
\leq \frac{2M\omega^2N_0}{|\sin(\tau\omega/2)|}
\leq \frac{2M\omega\bigl(\log(1/\omega)+\omega\bigr)}
{|\sin(\tau\omega/2)|}.
\end{equation}

To bound the tail in \eqref{eq:corr-split} when $N_0<N$, we use
$F(x)\leq 4x^2e^{-\pi x}\leq e^{-x}$ for $x>0$, obtaining
\[
\omega\left|\sum_{k=N_0+1}^{N}F(k\omega)\cos(k\omega\tau)\right|
\leq \omega\sum_{k=N_0+1}^{\infty}e^{-k\omega}
=e^{-N_0\omega}\frac{\omega}{e^\omega-1}.
\]
In this case,
$N_0=\lceil\omega^{-1}\log(1/\omega)\rceil$, so
$e^{-N_0\omega}\leq\omega$. Since $\omega/(e^\omega-1)<1$,
we conclude that
\begin{equation}\label{eq:tail}
\omega\left|\sum_{k=N_0+1}^{N}F(k\omega)\cos(k\omega\tau)\right|
\leq\omega.
\end{equation}
This bound trivially holds when the tail is empty.

Combining \eqref{eq:corr-split}, \eqref{eq:first-sum},
and \eqref{eq:tail}, we obtain
\[
|\Corr(\tau)|
\leq C_0\left(
\omega\alpha_0+
\frac{2M\omega\bigl(\log(1/\omega)+\omega\bigr)}
{|\sin(\tau\omega/2)|}
+\omega
\right)
\leq \frac{C\omega\log(1/\omega)}
{|\sin(\tau\omega/2)|},
\]
for sufficiently small $\omega>0$, as desired.
\end{proof}

We will need a special case of the well known Gershgorin circle theorem \cite{BrMe94} which we state as a lemma.  The result holds more generally for matrices that are not necessarily symmetric (where the intervals become disks in the complex plane).

\begin{lemma}[Gershgorin circle theorem for real symmetric matrices]\label{lemma:Gershgorin}
Let $A=(a_{ij})$ be a real symmetric $n\times n$ matrix. Then the eigenvalues of $A$ belong to the union of intervals
$
\bigcup_{j=1}^n [a_{jj}-R_j,\; a_{jj}+R_j],
$
where
$
R_j=\sum_{i\neq j}|a_{ij}|.
$
\end{lemma}

\begin{lemma}[Orthant probability for Gaussian vector]\label{lemma:orthant}
Let $X \sim \mathcal{N}(0,\Sigma)$ be a Gaussian vector in $\mathbb{R}^n$, and assume there exists a constant $c_1>0$ such that the covariance matrix $\Sigma$ has eigenvalues satisfying
\[
|\lambda_i - 1| \leq \frac{c_1}{n}.
\]
Then there exists $c_2>0$ such that for all $n$ sufficiently large
\[
\mathbb{P}(X_1 \ge 0,\dots,X_n \ge 0)
\;\le\;
c_2 2^{-n} .
\]
\end{lemma}

\begin{proof}
The orthant probability is given by
\[
\mathbb{P}(X \in \mathbb{R}_+^n)
=
\frac{1}{(2\pi)^{n/2}\sqrt{\det \Sigma}}
\int_{\mathbb{R}_+^n} \exp\!\left(-\tfrac12 x^T \Sigma^{-1} x\right)\,dx.
\]
Diagonalize $\Sigma = Q \Lambda Q^T$ with $Q$ orthogonal and $\Lambda = \mathrm{diag}(\lambda_1,\dots,\lambda_n)$, and make the change of variables $u = Q^T x$. Then
\[
\mathbb{P}(X \in \mathbb{R}_+^n)
=
\frac{1}{(2\pi)^{n/2}\sqrt{\det \Sigma}}
\int_{Q^T \mathbb{R}_+^n}
\exp\!\left(-\tfrac12 \sum_{i=1}^n \frac{u_i^2}{\lambda_i}\right)\,du.
\]
Let $\lambda_{\max} = \max_i \lambda_i$. Since $\lambda_i \le \lambda_{\max}$ for all $i$, we have
\[
\sum_{i=1}^n \frac{u_i^2}{\lambda_i}
\;\ge\;
\frac{1}{\lambda_{\max}} \|u\|^2,
\]
and hence
\[
\exp\!\left(-\tfrac12 \sum_{i=1}^n \frac{u_i^2}{\lambda_i}\right)
\;\le\;
\exp\!\left(-\frac{1}{2\lambda_{\max}} \|u\|^2\right).
\]
Therefore,
\[
\mathbb{P}(X \in \mathbb{R}_+^n)
\;\le\;
\frac{1}{(2\pi)^{n/2}\sqrt{\det \Sigma}}
\int_{Q^T \mathbb{R}_+^n}
\exp\!\left(-\frac{1}{2\lambda_{\max}} \|u\|^2\right)\,du.
\]
The integrand is invariant under orthogonal transformations, hence
\[
\int_{Q^T \mathbb{R}_+^n}
\exp\!\left(-\frac{1}{2\lambda_{\max}} \|u\|^2\right)\,du
=
\int_{\mathbb{R}_+^n}
\exp\!\left(-\frac{1}{2\lambda_{\max}} \|u\|^2\right)\,du =
2^{-n} (2\pi \lambda_{\max})^{n/2}.
\]
It follows that
\[
\mathbb{P}(X \in \mathbb{R}_+^n)
\;\le\;
2^{-n} \cdot \frac{\lambda_{\max}^{n/2}}{\sqrt{\det \Sigma}}.
\]
Using $\lambda_{\max} = 1 + O(1/n)$ and $\det \Sigma = \prod_{i=1}^n \lambda_i = \exp(O(1))$, we obtain
\[
\lambda_{\max}^{n/2} = \exp(O(1)), \qquad \sqrt{\det \Sigma} = \exp(O(1)).
\]
Therefore,
\[
\mathbb{P}(X \in \mathbb{R}_+^n)
\;\le\;
2^{-n} \exp(O(1)) = 2^{-n} O(1),
\]
as claimed.
\end{proof}

With these lemmas in hand, we proceed to the proof of the main result of this section.

\begin{proof}[Proof of Theorem \ref{thm:zero-free}]

Set $ n=\lfloor \omega^{-1/4}\rfloor$, and choose $n$ equally spaced points $t_j$ in the period interval $(-\pi/\omega, \pi/\omega)$ as $\{t_j\}_{j=1}^{n}=\{-\frac{\pi}{\omega}+j\Delta t\}_{j=1}^{n}$ where $\Delta t= {2\pi}{\omega^{-3/4}}$. 

Form the random vector of (variance-normalized) point evaluations
$$X := \left\{\frac{M(t_j)}{\sqrt{\mathbb{E}\left[ M(t_j)^2 \right]}}\right\}_{j=1}^n.$$
Here $M(t)$ is the Melnikov function for our system with random normal coefficients, making this a Gaussian vector whose entries are standard normals.

By the intermediate value theorem, the event that the Melnikov function $M(t)$ is zero-free implies that either all entries of $X$ are positive or all entries are negative.
Hence, the probability that $M(t)$ is zero-free is at most twice the orthant probability of $X$.  This reduces the proof of the theorem to showing that there exists $c_2>0$ such that
\begin{equation}
\label{eq:orthantgoal}
\mathbb{P}(X_1 \ge 0,\dots,X_n \ge 0)\leq c_2 2^{-n},
\end{equation}
for all $n$ sufficiently large.

Let $\tau_{i,j}= t_i-t_j$. Then $X\sim N(0, \Sigma)$ where
% \begin{center}
% $\Sigma=
% \begin{bmatrix}
% \Corr(t_1,t_1) & \dots & \Corr(t_1, t_i) & \dots & \Corr(t_1, t_n)\\
% \vdots & \ddots & \vdots &  &\vdots\\
% \Corr(t_i, t_1) &\dots & \Corr(t_i, t_i) & \dots & \Corr(t_i, t_n) \\
% \vdots &  & \vdots &\ddots& \vdots\\
% \Corr(t_n, t_1) & \dots &\Corr(t_n, t_i)&\dots & \Corr(t_n, t_n) 
% \end{bmatrix}
% $
% \end{center}

\begin{center}
$\Sigma=
\begin{bmatrix}
1 & \dots & \Corr(\tau_{1,i}) & \dots & \Corr(\tau_{1,n})\\
\vdots & \ddots & \vdots &  &\vdots\\
\Corr(\tau_{i,1}) &\dots & 1 & \dots & \Corr(\tau_{i,n}) \\
\vdots &  & \vdots &\ddots& \vdots\\
\Corr(\tau_{n,1}) & \dots &\Corr(\tau_{n,i})&\dots & 1 
\end{bmatrix}.
$
\end{center}
Note that 
$$\tau_{i,j}=t_i-t_j=\Delta t (i-j)=2\pi \omega^{-3/4}(i-j).$$
Then applying Lemma \ref{lemma:corrinbulk}, we get 
$$|\Corr(\tau_{i,j})|\leq  \frac{C}{|\sin(\omega \tau_{i,j}/2)|}\omega \log(1/\omega)=\frac{C}{|\sin(\pi\omega^{1/4}(i-j))|}\omega \log(1/\omega).$$
For all $i \neq j$ the value in the denominator is smallest when $i-j=\pm 1$.

For small $\omega$, $\sin(\pi\omega^{1/4})\geq \omega^{1/4}$, and we obtain the following uniform estimate for the off-diagonal entries of $\Sigma$
\begin{align*}
|\Corr(\tau_{i,j})|&\leq  \frac{C}{|\sin(\pi\omega^{1/4}(i-j))|}\omega \log(1/\omega)\leq {C}\omega^{3/4} \log(1/\omega)\leq \hat{C}\omega^{1/2}, \quad i \neq j.
\end{align*}
This immediately gives the estimate 
$$R_j := \sum_{i\neq j}|\Corr(\tau_{ij})|\leq \hat{C} \omega^{1/4},$$ and hence, applying Lemma \ref{lemma:Gershgorin}, we get that the eigenvalues $\lambda_i$ of $\Sigma$  are $\lambda_i=1+O(1/n)$. Then \eqref{eq:orthantgoal} follows by Lemma \ref{lemma:orthant}.
This concludes the proof of the theorem.
\end{proof}

\section{Concluding Remarks}\label{sec:concl}

\subsection{Positive damping}
Although our Gaussian model allows both signs of the damping
coefficient $\delta$, all of our zero-count results remain valid
when conditioning on $\delta>0$. Indeed, simultaneously reversing
the signs of $\delta$ and all forcing coefficients preserves their
joint distribution and replaces the Melnikov function $M$ by $-M$,
leaving its zeros unchanged while exchanging the events $\delta>0$
and $\delta<0$. Equivalently, we may take $\delta$
to have a positive half-normal distribution, independently of
the standard Gaussian forcing coefficients.

\subsection{Other probability distributions} A natural direction for future research is to consider other probability distributions on parameter space besides i.i.d. Gaussian coefficients, such as allowing the variances of the coefficients $a_k, b_k$ to depend on $k$ and $N$, allowing non-centered Gaussians, allowing dependence in the coefficients, or allowing non-Gaussian random variables.

\subsection{Other nonlinear oscillators} 
Our methods extend to other perturbed Hamiltonian systems.  For instance, let us consider the following example from \cite{Rand} with quadratic nonlinearity
\begin{equation}\label{eq:squaredeq}
\ddot{x}-x-\frac{x^2}{2}=\varepsilon(f(t)-\delta\dot{x}),
\end{equation}
where as before we take the forcing function $f(t)=f_{N,\omega}(t)$ to be a random trigonometric polynomial. We summarize, with details omitted, the results obtained by applying the same methods from the previous sections.

The Melnikov function associated with the homoclinic orbit $q^0(t)$ of \eqref{eq:squaredeq} is given by (after scaling by a constant factor)
\begin{equation}\label{eq:Melnikovfunction_rand}
M_{N, \omega}(t_0)=-\frac{2}{5\pi}\delta+\sum_{k=1}^N (\omega k)^2\csch(\pi k \omega)[b_k\cos(\omega k t_0)-a_k\sin(\omega k t_0)].
\end{equation}
An application of the Kac-Rice formula (Theorem \ref{thm:RandTrigWeights}) gives the expected number of zeros over a single period
\begin{equation}\label{eq:Melnikovzeros}
\EE Z_{N, \omega}=\frac{2}{\omega}\cdot \sqrt{\frac{\sum_{k=1}^{N} (\omega k)^6\csch^2(\pi k\omega)}{4/(25\pi^2)+\sum_{k=1}^{N}(\omega k)^4\csch^2(\pi k \omega)}}.
\end{equation}
Letting $C_\omega$ denote, as before, the limiting expectation as $N$ approaches infinity, and subsequently taking $\omega\rightarrow 0$, we find the same order of growth but with a different constant $\displaystyle C_\omega \sim \sqrt{\frac{5}{7}} \cdot {\frac{2}{\omega}}$ as $\omega \rightarrow 0$.

The conclusion of Theorem \ref{thm:nozeros} still holds, i.e., the Melnikov function has no zeros with overwhelming probability as $\omega \rightarrow \infty$, and
the conclusion of Theorem \ref{thm:zero-free} holds as well, giving overwhelming probability of $M_{N,\omega}$ having zeros as $\omega \rightarrow 0$ (under the same assumption that $N\omega>\gamma$ for fixed $\gamma>0$).
This follows from applying the same method of proof with small modifications.  The key observation is that the correlation function of $M_{N, \omega}$, which is given by 
$$\Corr(\tau)=\frac{4/(25\pi^2)+\sum_{k=1}^N(k \omega)^4\, \,\mathrm{csch}^2\!\big(\pi k \omega\big)\cos(\omega k\tau)}{4/(25\pi^2)+\sum_{k=1}^N (k \omega)^4 \,\mathrm{csch}^2\!\big(\pi k \omega\big) },$$
enjoys the same smoothness and decay rate needed to establish Lemma \ref{lemma:corrinbulk}.

\subsection{Generalizations to classes of perturbed Hamiltonian systems} As one might suspect from the outcomes in the above example, some of our results extend to general classes of perturbed Hamiltonian systems, further expanding the broad perspective offered by a probabilistic approach.   For instance, under suitable assumptions on the Hamiltonian $H$, we have the following general result giving overwhelming probability of no zeros at high frequency.

\begin{thm}[Absence of Melnikov zeros at high frequency]
\label{thm:generalhighfrequency}
Let $H:\mathbb{R}^2\to\mathbb{R}$ be a polynomial Hamiltonian
whose vector field $h=(H_v,-H_u)$ has a hyperbolic saddle
at the origin with a homoclinic orbit
$q^0(t)=(u^0(t),v^0(t))$. Assume that
\be\label{eq:AssumeNotZero}
\int_{-\infty}^{\infty}
H_v(q^0(t))v^0(t)\,dt\neq 0.
\ee
Let $M_{N,\omega}$ denote the Melnikov function associated to the perturbed system
\[
\begin{cases}
\dot u=H_v(u,v),\\
\dot v=-H_u(u,v)
+\varepsilon\bigl(f_{N,\omega}(t)-\delta v\bigr),
\end{cases}
\quad
f_{N,\omega}(t)
=\sum_{k=1}^N
\bigl[a_k\cos(k\omega t)+b_k\sin(k\omega t)\bigr],
\]
where $\delta,a_1,b_1,\ldots,a_N,b_N$ are independent
standard Gaussian random variables.
Let $Z_{N,\omega}$ denote the number of zeros of $M_{N,\omega}$ over a single period $[0,2\pi/\omega)$.
For every $A>0$, there exists a constant $C_A>0$, depending only on $A$, $H$, and the chosen homoclinic orbit, such that for all $\omega\geq 1$ and $N\geq 1$,
\be\label{eq:wop}
\mathbb{P}(Z_{N,\omega}\geq 1)
\leq C_A\omega^{-A}.
\ee
In particular, $M_{N,\omega}$ is zero-free
with overwhelming probability as $\omega\to\infty$.
\end{thm}
The proof uses a well-known property of the Fourier transform, namely, that it is an automorphism of the Schwartz space of rapidly decreasing functions \cite[Ch.~5, Cor.~1.10]{SteinShakarchiFourier}.
\begin{proof}
Let $\lambda_s<0<\lambda_u$ be the stable and unstable eigenvalues of the linearization at the origin, and take
$0<\alpha<\min\{|\lambda_s|,\lambda_u\}$.

By the stable and unstable manifold theorem, the homoclinic orbit $q^0(t)$ converges exponentially to the origin forward and backward in time, and hence $|q^0(t)|=O(e^{-\alpha |t|})$ as $t\to\pm\infty$.

Writing
\[
\nabla H(x)=
\begin{bmatrix}
p_1(x)\\
p_2(x)
\end{bmatrix},
\]
since $H$ is assumed polynomial, $p_1,p_2$ are polynomials.
Since the origin is an equilibrium of the Hamiltonian vector field, we have $\nabla H(0)=0$.
In particular, letting $F(t)=p_2(q^0(t))$,
we have $F(t)=O(e^{-\alpha |t|})$ as $t\to\pm\infty$.
Moreover, since $q^0$ is a smooth orbit of the smooth Hamiltonian vector field, we have $\frac{d}{dt}q^0(t)=h(q^0(t))=O(e^{-\alpha|t|})$.
Repeated differentiation using the chain rule shows that every derivative of $F$ also decays exponentially. Thus, $F$ belongs to the Schwartz class of smooth functions all of whose derivatives are rapidly decreasing.

Introducing the complex numbers $c_k=a_k-i b_k$, the Melnikov integral can be written in terms of the Fourier transform of $F$,
\[
\widehat F(s):=\int_{-\infty}^{\infty}F(t)e^{ist}\,dt,
\]
as follows:
\[
M_{N,\omega}(t_0)
=
-\delta D+
\Real\left(
\sum_{k=1}^N
\widehat F(k\omega)c_k e^{ik\omega t_0}
\right),
\qquad
D:=\int_{-\infty}^{\infty}
p_2(q^0(t))v^0(t)\,dt,
\]
where $\Real(\cdot)$ denotes the real part.

By the Kac--Rice formula,
\[
\EE Z_{N,\omega}
=
2\sqrt{
\frac{\sum_{k=1}^N k^2|\widehat F(k\omega)|^2}
{D^2+\sum_{k=1}^N|\widehat F(k\omega)|^2}
}.
\]
If all nonconstant Fourier coefficients vanish, the Melnikov function is almost surely a nonzero constant, and the expectation is zero.

It suffices to prove the estimate \eqref{eq:wop} for $A>2$.
Since $\widehat F$ belongs to the Schwartz class \cite[Ch.~5, Cor.~1.10]{SteinShakarchiFourier}, there exists $K_A>0$ such that
$|\widehat F(s)|\leq K_As^{-A}$ for $s\geq1$.
Since $D\neq0$ by assumption \eqref{eq:AssumeNotZero}, we have uniformly in $N\geq1$ and $\omega\geq1$,
\[
\EE Z_{N,\omega}
\leq
\frac{2K_A}{|D|}\omega^{-A}
\left(\sum_{k=1}^{\infty}k^{2-2A}\right)^{1/2}
\leq C_A\omega^{-A}.
\]
The conclusion of the theorem then follows, as in the proof of Theorem \ref{thm:nozeros}, from Markov's inequality. \end{proof}

Exploring further generalizations of our results to classes of perturbed Hamiltonian systems seems a worthwhile direction for future research.

\subsection{Persistence probability} The probability of a random trigonometric polynomial having no zeros is closely related to the \emph{persistence probability} for Gaussian stationary processes (GSPs) \cite{Sumit2017}.  Namely, the probability of no zeros is twice the persistence probability (since the zero-free event consists of the two mutually exclusive and equiprobable events that the process is everywhere positive or everywhere negative).
In particular, by relating the random Melnikov function to its rescaling limit (this would require a coupling argument and continuity estimates), it may be possible to improve Theorem \ref{thm:zero-free} and obtain the sharp decay rate for the zero-free event using methods from \cite{Sumit2025}.  In this direction, we note that the limiting correlation function, (first taking $N \rightarrow \infty$ followed by $\omega \rightarrow 0$), is given by 
\begin{align*}
\Corr(\tau)
&=\frac{\int_0^\infty x^2\cdot\sech^2(\frac{\pi}{2}x)\cos(x \tau) dx}{\int_0^\infty  x^2\cdot\sech^2(\frac{\pi}{2}x) dx} \\
&=3 \csch(\tau) \left( 2\coth(\tau)-2\tau \csch^2(\tau)-\tau \right).
\end{align*}
The spectral measure, defined in \cite[Introduction]{Sumit2025}, associated to this correlation function is absolutely continuous with respect to Lebesgue measure, and its density is proportional to $x^2\cdot\sech^2(\frac{\pi}{2}x)$ which vanishes at $x=0$.  While this behavior falls outside the class of examples addressed by the main result in \cite{Sumit2025}, this exceptional case is discussed in \cite[Remark 2]{Sumit2025}, where it is noted that the so-called persistence exponent turns out to be infinite.  This indicates that the persistence probability (of the GSP obtained as a rescaling limit of the random periodic function) over a long interval decays faster than exponential in the size of the interval.  Heuristically, this long interval estimate should give information on the persistence probability for the original random periodic function over its growing period, which suggests that the decay rate in Theorem \ref{thm:zero-free} can be improved to decay that is faster than exponential.

\subsection{Quantitative perturbation thresholds}
The dynamical conclusions drawn from our results (invoking the classical framework summarized in Section \ref{sec:prelim}) hold for
sufficiently small $\e>0$, with a threshold that may depend on
the sampled coefficients as well as $N$ and $\omega$.

A natural direction for future research is to obtain
probabilistic estimates for this threshold. The quantitative
version of Bulinskaya's lemma provided in
\cite[Lemma~7]{NazarovSodin2} controls the probability that a
Gaussian function and its derivative are simultaneously small
somewhere. Such estimates could provide one ingredient in
obtaining explicit ranges of $\e>0$ for which the dynamical
conclusions remain valid with high probability. A substantial
additional challenge is to obtain suitable quantitative bounds
on the remainder in the perturbation expansion, of which the
Melnikov integral provides the first-order term.

\subsection*{Acknowledgments and statement on the use of AI}
The second-named author acknowledges support from the Simons Foundation grant 712397.
%and also acknowledges OpenAI for a grant of access to models through the ChatGPT for Academic Researchers program.
The results and their proofs were conceived by the authors.
ChatGPT 6 was used to assist with proofreading and revising
the completed manuscript. The authors evaluated each suggestion and made their own revisions.

\section{Appendix}

Here we provide details for the evaluation of the integrals \eqref{eq:residue1} and \eqref{eq:residue2} using contour integration and residues.

\subsection{Residue computation}\label{sec:residues}
Our goal is to show
\begin{equation}\label{eq:Residue}
\int_{-\infty}^{\infty} \sech(t)\tanh(t) e^{i\omega(t+t_0)}\, dt
= 
i \pi \omega e^{i\omega t_0}
\sech\!\left(\frac{\pi\omega}{2}\right), \quad \omega > 0,
\end{equation}
from which \eqref{eq:residue1} and \eqref{eq:residue2} follow by taking real and imaginary parts.

For a positive integer $\ell$, consider the rectangular contour 
$\Gamma_{\ell} = \gamma_1 + \gamma_2 + \gamma_3 + \gamma_4$ in the complex plane, where
\begin{align*}
\gamma_1 &: [-\ell, \ell],\\
\gamma_2 &: [\ell, \ell + i 2\pi \ell],\\
\gamma_3 &: [\ell + i 2\pi \ell, -\ell + i 2\pi \ell],\\
\gamma_4 &: [-\ell + i 2\pi \ell, -\ell].
\end{align*}
Note that our integral in \eqref{eq:Residue} is the integral along $\gamma_1$ as $\ell \to \infty$. 
Moreover, we will show (see Section \ref{sec:gammas} below) that the integrals over $\gamma_2, \gamma_3$, and $\gamma_4$ tend to zero as $\ell \rightarrow \infty$.

Hence, 
\begin{equation}
\int_{-\infty}^{\infty} \sech(t)\tanh(t)e^{i\omega(t+t_0)}\, dt
=  \lim_{\ell\to\infty} \oint_{\Gamma_{\ell}} \sech(z)\tanh(z)e^{i\omega(z+t_0)}\, dz .
\end{equation}
Since the integrand
\[
\sech(z)\tanh(z)e^{i\omega(z+t_0)} 
= \frac{\sinh(z)}{\cosh^2(z)} e^{i\omega(z+t_0)},
\]
has poles of order two at 
\( z_k = i\left(\frac{\pi}{2} + k\pi\right), \quad k \in \mathbb{Z} \), we have
\begin{align*}
\lim_{\ell\to\infty} \oint_{\Gamma_{\ell}} \frac{\sinh(z)}{\cosh^2(z)} e^{i\omega(z+t_0)}\, dz
&= 2\pi i \sum_{k=0}^{\infty} 
\operatorname{Res}\left(
\frac{\sinh(z)}{\cosh^2(z)} e^{i\omega(z+t_0)}; z_k
\right).
\end{align*}

To compute these residues, define for each $k$:
\[
\cosh(z) = (z - z_k) g_k(z),
\quad \text{where} \quad 
z_k = i\left(\frac{\pi}{2} + k\pi\right).
\]
Then
\[
\lim_{z\to z_k} g_k(z) = \sinh(z_k) = 
\begin{cases}
i, & k \text{ even},\\
-i, & k \text{ odd},
\end{cases}
\]
and 
$$\lim_{z\rightarrow z_k}g_k'(z)=\lim_{z\rightarrow z_k} \frac{\sinh(z)(z-z_k)-\cosh(z)}{(z-z_k)^2}=\lim_{z\rightarrow z_k}\frac{\cosh(z)}{2}=0.$$

Thus,
\begin{align*}
\sum_{k=0}^{\infty} 
\operatorname{Res}\left(
\frac{\sinh(z)}{\cosh^2(z)} e^{i\omega(z+t_0)}; z_k
\right)
&= 
\sum_{k=0}^{\infty} 
\left[
\frac{d}{dz}
\left(
\frac{\sinh(z)e^{i\omega(z+t_0)}}{g_k(z)^2}
\right)
\right]_{z=z_k} \\
&= 
\sum_{k=0}^{\infty} 
\frac{i\omega\sinh(z_k)e^{i\omega(z_k+t_0)}}{g_k(z_k)^2} \\
&= 
\omega e^{i\omega t_0}\sum_{k=0}^{\infty} (-1)^k e^{i\omega z_k} \\
&= 
\omega e^{i\omega t_0} e^{-\frac{\pi\omega}{2}}
\cdot \frac{1}{1 + e^{-\pi\omega}} \\
&= 
\frac{\omega}{2} e^{i\omega t_0} \sech\!\left(\frac{\pi\omega}{2}\right).
\end{align*}

Therefore,
\be
\int_{-\infty}^{\infty} \sech(t)\tanh(t) e^{i\omega(t+t_0)} \, dt
=
2\pi i \cdot \frac{\omega}{2} e^{i\omega t_0}
\sech\!\left(\frac{\pi\omega}{2}\right),
\ee
which verifies \eqref{eq:Residue}.

It remains only to show that the integrals over $\gamma_2$, $\gamma_3$, and $\gamma_4$ tend to zero as $\ell \to \infty$.

\subsection{Estimates along remaining segments of the contour}\label{sec:gammas}

For $z=x+iy$, we have
\[
F(z) := \frac{\sinh(z)}{\cosh^2(z)} e^{i\omega(z+t_0)}
= \frac{\sinh(x)\cos(y) + i\cosh(x)\sin(y)}
{[\cosh(x)\cos(y) + i\sinh(x)\sin(y)]^2}
e^{-\omega y} e^{i\omega(x+t_0)},
\]
which we use to estimate the modulus of the integrand
\[
\bigg|\frac{\sinh(z)}{\cosh^2(z)} e^{i\omega(z+t_0)}\bigg|
= \frac{\sqrt{\sinh^2(x) + \sin^2(y)}}{\cosh^2(x) - \sin^2(y)}\, e^{-\omega y}.
\]

On $\gamma_3$ (where $y = 2\pi \ell$ and $x \in [-\ell, \ell]$), we have
\[
\Big|\int_{\gamma_3} F(z)\,dz\Big|
\le 2 \ell \, e^{- 2 \pi \omega \ell} 
\max_{x\in[-\ell,\ell]}
\frac{|\sinh(x)|}{\cosh^2(x)} \le 2\ell\, e^{-2\pi\omega\ell},
\]
which tends to zero as $\ell \to \infty$.
Similar estimates show that the integral over $\gamma_2$ satisfies
\[
\Big|\int_{\gamma_2} F(z)\,dz\Big|
\le 2\pi \ell 
\frac{\sqrt{\sinh^2(\ell)+1}}{\cosh^2(\ell)-1} = 2\pi \ell 
\frac{\cosh(\ell)}{\cosh^2(\ell)-1},
\]
which tends toward zero as $\ell \to \infty$ (the integral along $\gamma_4$ satisfies the same estimate).

\bibliographystyle{abbrv}
\bibliography{Melnikov}

@book {Stein,
    AUTHOR = {Stein, Elias M. and Shakarchi, Rami},
     TITLE = {Real analysis},
    SERIES = {Princeton Lectures in Analysis},
    VOLUME = {3},
      NOTE = {Measure theory, integration, and Hilbert spaces},
 PUBLISHER = {Princeton University Press, Princeton, NJ},
      YEAR = {2005},
     PAGES = {xx+402},
      ISBN = {0-691-11386-6},
   MRCLASS = {28B15 (00-01 26-01 26A03 26A09 28-01 46-01 47-01)},
  MRNUMBER = {2129625},
MRREVIEWER = {Steven\ George\ Krantz},
}

@book {Rand,
    AUTHOR = {Rand, Richard H.},
     TITLE = {Topics in nonlinear dynamics with computer algebra},
    SERIES = {Computation in Education: Mathematics, Science and
              Engineering},
    VOLUME = {1},
 PUBLISHER = {Gordon and Breach Science Publishers, Langhorne, PA},
      YEAR = {1994},
     PAGES = {x+229}
}

@ARTICLE{KLN,
    AUTHOR = {Krishnapur, Manjunath and Lundberg, Erik and Nguyen, Oanh},
     TITLE = {The number of limit cycles bifurcating from a randomly
              perturbed center},
   JOURNAL = {Ann. Probab.},
  FJOURNAL = {The Annals of Probability},
    VOLUME = {54},
      YEAR = {2026},
    NUMBER = {1},
     PAGES = {269--300}
}

@article {HolmesStrange,
    AUTHOR = {Holmes, P.},
     TITLE = {A nonlinear oscillator with a strange attractor},
   JOURNAL = {Philos. Trans. Roy. Soc. London Ser. A},
  FJOURNAL = {Philosophical Transactions of the Royal Society of London.
              Series A. Mathematical and Physical Sciences},
    VOLUME = {292},
      YEAR = {1979},
    NUMBER = {1394},
     PAGES = {419--448}
}

@article {PiecewiseLinear,
    AUTHOR = {Vallipriyatharsini, S. and Bazeera, A. Z. and Chinnathambi, V.
              and Rajasekar, S.},
     TITLE = {Homoclinic transition to chaos in the {D}uffing oscillator
              driven by periodic piecewise linear forces},
   JOURNAL = {J. Nonlinear Anal. Optim.},
  FJOURNAL = {Journal of Nonlinear Analysis and Optimization. Theory and
              Applications},
    VOLUME = {13},
      YEAR = {2022},
    NUMBER = {1},
     PAGES = {25--38},
      ISSN = {1906-9685,3057-0867},
   MRCLASS = {37D45 (34A08 34D10 37C29 37J20)},
  MRNUMBER = {4431662},
}

@book {GuckHolmes,
    AUTHOR = {Guckenheimer, John and Holmes, Philip},
     TITLE = {Nonlinear oscillations, dynamical systems, and bifurcations of
              vector fields},
    SERIES = {Applied Mathematical Sciences},
    VOLUME = {42},
 PUBLISHER = {Springer-Verlag, New York},
      YEAR = {1983},
     PAGES = {xvi+453},
      ISBN = {0-387-90819-6},
   MRCLASS = {58-02 (34C35 58Fxx 70K05)},
  MRNUMBER = {709768},
MRREVIEWER = {Igor\ Gumowski},
       DOI = {10.1007/978-1-4612-1140-2},
       URL = {https://doi.org/10.1007/978-1-4612-1140-2},
}

@article {EdelmanKostlan95,
    AUTHOR = {Edelman, Alan and Kostlan, Eric},
     TITLE = {How many zeros of a random polynomial are real?},
   JOURNAL = {Bull. Amer. Math. Soc. (N.S.)},
  FJOURNAL = {American Mathematical Society. Bulletin. New Series},
    VOLUME = {32},
      YEAR = {1995},
    NUMBER = {1},
     PAGES = {1--37},
      ISSN = {0273-0979},
   MRCLASS = {60G99 (30B20 42A05)},
  MRNUMBER = {1290398},
MRREVIEWER = {Kambiz Farahmand},
       DOI = {10.1090/S0273-0979-1995-00571-9},
       URL = {http://dx.doi.org/10.1090/S0273-0979-1995-00571-9},
}

@article {Arnold,
    AUTHOR = {Arnold, V. I.},
     TITLE = {Instability of dynamical systems with many degrees of freedom},
   JOURNAL = {Dokl. Akad. Nauk SSSR},
  FJOURNAL = {Doklady Akademii Nauk SSSR},
    VOLUME = {156},
      YEAR = {1964},
     PAGES = {9--12},
      ISSN = {0002-3264},
   MRCLASS = {34.65 (57.48)},
  MRNUMBER = {163026},
MRREVIEWER = {J.\ Moser},
}

@article{Yang,
  title = {Duffing equation with two periodic forcings: The phase effect},
  author = {Yang, Junzhong and Qu, Zhilin and Hu, Gang},
  journal = {Phys. Rev. E},
  volume = {53},
  issue = {5},
  pages = {4402--4413},
  numpages = {0},
  year = {1996}
}

@Article{Li2024,
author={Li, Shuangbao
and Xu, Rui
and Kou, Liying},
title={Suppressing homoclinic chaos for a class of vibro-impact oscillators by non-harmonic periodic excitations},
journal={Nonlinear Dynamics},
year={2024},
month={Jul},
day={01},
volume={112},
number={13},
pages={10845-10870}
}

@article{Wang,
title = {Chaos control in duffing system},
author = {Ruiqi Wang and Jin Deng and Zhujun Jing},
journal = {Chaos, Solitons \& Fractals},
volume = {27},
number = {1},
pages = {249-257},
year = {2006}
}

@article {Sumit2025,
    AUTHOR = {Feldheim, Naomi D. and Feldheim, Ohad N. and Mukherjee, Sumit},
     TITLE = {Persistence and ball exponents for {G}aussian stationary
              processes},
   JOURNAL = {Comm. Pure Appl. Math.},
  FJOURNAL = {Communications on Pure and Applied Mathematics},
    VOLUME = {78},
      YEAR = {2025},
    NUMBER = {10},
     PAGES = {1949--2000}
}

@article {Sumit2017,
    AUTHOR = {Dembo, Amir and Mukherjee, Sumit},
     TITLE = {Persistence of {G}aussian processes: non-summable
              correlations},
   JOURNAL = {Probab. Theory Related Fields},
  FJOURNAL = {Probability Theory and Related Fields},
    VOLUME = {169},
      YEAR = {2017},
    NUMBER = {3-4},
     PAGES = {1007--1039},
      ISSN = {0178-8051,1432-2064},
   MRCLASS = {60G15 (82C24)},
  MRNUMBER = {3719062},
MRREVIEWER = {Mikhail\ A.\ Lifshits},
       DOI = {10.1007/s00440-016-0746-9},
       URL = {https://doi.org/10.1007/s00440-016-0746-9},
}

@article{Chacon,
  title = {Routes to suppressing chaos by weak periodic perturbations},
  author = {Chac\'on, R. and D\'{\i}az Bejarano, J.},
  journal = {Phys. Rev. Lett.},
  volume = {71},
  issue = {19},
  pages = {3103--3106},
  numpages = {0},
  year = {1993},
  month = {Nov},
  publisher = {American Physical Society},
  doi = {10.1103/PhysRevLett.71.3103},
  url = {https://link.aps.org/doi/10.1103/PhysRevLett.71.3103}
}

@article{Gear,
title = {An analytical study of controlling chaotic dynamics in a spur gear system},
author = {A. Saghafi and A. Farshidianfar},
journal = {Mechanism and Machine Theory},
volume = {96},
pages = {179-191},
year = {2016}
}

@article {Melnikov,
    AUTHOR = {Melnikov, V. K.},
     TITLE = {On the stability of a center for time-periodic perturbations},
   JOURNAL = {Trudy Moskov. Mat. Ob\v s\v c.},
  FJOURNAL = {Trudy Moskovskogo Matemati\v ceskogo Ob\v s\v cestva},
    VOLUME = {12},
      YEAR = {1963},
     PAGES = {3--52}
}

@book{SteinShakarchiFourier,
  author    = {Elias M. Stein and Rami Shakarchi},
  title     = {Fourier Analysis: An Introduction},
  series    = {Princeton Lectures in Analysis},
  volume    = {1},
  publisher = {Princeton University Press},
  year      = {2003}
}

@article {BrMe94,
    AUTHOR = {Brualdi, Richard A. and Mellendorf, Stephen},
     TITLE = {Regions in the complex plane containing the eigenvalues of a
              matrix},
   JOURNAL = {Amer. Math. Monthly},
  FJOURNAL = {American Mathematical Monthly},
    VOLUME = {101},
      YEAR = {1994},
    NUMBER = {10},
     PAGES = {975--985},
      ISSN = {0002-9890,1930-0972},
   MRCLASS = {15A18},
  MRNUMBER = {1304322},
       DOI = {10.2307/2975164},
       URL = {https://doi.org/10.2307/2975164},
}

@book {Grad,
    AUTHOR = {Gradshteyn, I. S. and Ryzhik, I. M.},
     TITLE = {Table of integrals, series, and products},
   EDITION = {Eighth},
      NOTE = {Translated from the Russian,
              Translation edited and with a preface by Daniel Zwillinger and
              Victor Moll,
              Revised from the seventh edition},
 PUBLISHER = {Elsevier/Academic Press, Amsterdam},
      YEAR = {2015},
     PAGES = {xlvi+1133}
}

@article {Lundberg,
    AUTHOR = {Lundberg, Erik},
     TITLE = {Limit cycle enumeration in random vector fields},
   JOURNAL = {Trans. Amer. Math. Soc.},
  FJOURNAL = {Transactions of the American Mathematical Society},
    VOLUME = {376},
      YEAR = {2023},
    NUMBER = {8},
     PAGES = {5693--5730}
}

@Article{AzaisBul,
author={Aza{\"i}s, Jean-Marc
and Wschebor, Mario},
title={On the regularity of the distribution of the maximum of one-parameter Gaussian processes},
journal={Probability Theory and Related Fields},
year={2001},
month={Jan},
day={01},
volume={119},
number={1},
pages={70-98}
}

@article {NazarovSodin2,
    AUTHOR = {Nazarov, F. and Sodin, M.},
     TITLE = {Asymptotic laws for the spatial distribution and the number of
              connected components of zero sets of {G}aussian random
              functions},
   JOURNAL = {Zh. Mat. Fiz. Anal. Geom.},
  FJOURNAL = {Zhurnal Matematichesko\u\i \ Fiziki, Analiza, Geometrii. Journal
              of Mathematical Physics, Analysis, Geometry},
    VOLUME = {12},
      YEAR = {2016},
    NUMBER = {3},
     PAGES = {205--278},
      ISSN = {1812-9471},
   MRCLASS = {60G15},
  MRNUMBER = {3522141},
MRREVIEWER = {Giacomo Aletti},
       DOI = {10.15407/mag12.03.205},
       URL = {http://dx.doi.org/10.15407/mag12.03.205},
}

@book {AzaisWscheborbook,
    AUTHOR = {Azais, Jean-Marc and Wschebor, Mario},
     TITLE = {Level sets and extrema of random processes and fields},
 PUBLISHER = {John Wiley \& Sons, Inc., Hoboken, NJ},
      YEAR = {2009},
     PAGES = {xii+393},
      ISBN = {978-0-470-40933-6},
   MRCLASS = {60-02 (60E15 60G05 60G15 60G60 60G70)},
  MRNUMBER = {2478201},
MRREVIEWER = {Anna Amirdjanova},
       DOI = {10.1002/9780470434642},
       URL = {http://dx.doi.org/10.1002/9780470434642},
}

@book{bruckner1997real,
  title={Real Analysis},
  author={Bruckner, A.M. and Bruckner, J.B. and Thomson, B.S.},
  isbn={9780134588865},
  lccn={96022123},
  url={https://books.google.com/books?id=1WY6u0C_jEsC},
  year={1997},
  publisher={Prentice-Hall}
}

@book {BabyRudin,
    AUTHOR = {Rudin, Walter},
     TITLE = {Principles of mathematical analysis},
    SERIES = {International Series in Pure and Applied Mathematics},
   EDITION = {Third},
 PUBLISHER = {McGraw-Hill Book Co., New York-Auckland-D\"usseldorf},
      YEAR = {1976},
     PAGES = {x+342},
   MRCLASS = {26-02},
  MRNUMBER = {385023},
}

\end{document}